\documentclass[a4paper,12pt]{amsart}
\usepackage[colorlinks,linkcolor=blue,anchorcolor=Periwinkle,
    citecolor=blue,urlcolor=Emerald]{hyperref}
\usepackage[all]{xy}
\usepackage{stmaryrd}
\usepackage{hyperref}
\usepackage{mathabx}
\usepackage{cite}
\usepackage{qtree}
\usepackage{amssymb}
\usepackage{amsthm}
\usepackage{amsmath} 
\usepackage{mathtools}
\usepackage{graphicx}
\usepackage{adjustbox}
\usepackage{xcolor}
\hypersetup{colorlinks=false,allbordercolors=blue,pdfborderstyle={/S/U/W 1}}
\usepackage{epsfig}
\usepackage{indentfirst}
\usepackage[usenames,dvipsnames]{pstricks}
\usepackage{mathrsfs}
\usepackage{tikz}
\usetikzlibrary{arrows.meta}
\SelectTips{cm}{}
\usepackage[mathscr]{euscript}
\usepackage{listings}
\usepackage{young}
\usepackage{ytableau}
\usepackage{enumitem}
\usetikzlibrary{arrows}
\usepackage{indentfirst}
\usepackage[usenames,dvipsnames]{pstricks}
\usepackage{pst-grad}
\usepackage{pst-plot}
\usepackage{mathrsfs}
\usepackage{subcaption}
\usepackage{youngtab}
\usepackage{tikz-cd}
\usepackage{ifthen}
\usepackage{color}
\usepackage[normalem]{ulem}
\usetikzlibrary{arrows}
\usepackage{quiver}
\usetikzlibrary{automata}
\usepackage{ragged2e}
\usepackage{bbm}
\usepackage{tikz}
\usetikzlibrary{arrows.meta}

\newtheorem{theorem}{Theorem}[section]
\newtheorem{corollary}[theorem]{Corollary}
\newtheorem{definition}[theorem]{Definition}
\newtheorem{example}[theorem]{Example}
\newtheorem{lemma}[theorem]{Lemma}

\newtheorem{proposition}[theorem]{Proposition}

\newtheorem{question}[theorem]{Question}
\newtheorem{remark}[theorem]{Remark}

\newtheorem{conjecture*}[theorem]{Conjecture}

\newtheorem{innercustomgeneric}{\customgenericname}
\providecommand{\customgenericname}{}
\newcommand{\newcustomtheorem}[2]{%
  \newenvironment{#1}[1]
  {%
   \renewcommand\customgenericname{#2}%
   \renewcommand\theinnercustomgeneric{##1}%
   \innercustomgeneric
  }
  {\endinnercustomgeneric}
}

\newcustomtheorem{customtheorem}{\bf Theorem}
\newcustomtheorem{customlemma}{Lemma}

\DeclareMathOperator{\Hom}{Hom}
\DeclareMathOperator{\Ext}{Ext}
\DeclareMathOperator{\Tor}{Tor}
\DeclareMathOperator{\End}{End}
\DeclareMathOperator{\im}{im}
\DeclareMathOperator{\pd}{pd}
\DeclareMathOperator{\id}{id}
\DeclareMathOperator{\fp}{FP}
\DeclareMathOperator{\CF}{CF}
\DeclareMathOperator{\IF}{IF}
\DeclareMathOperator{\F}{F}
\DeclareMathOperator{\Mod}{Mod}
\DeclareMathOperator{\Add}{Add}

\DeclareMathOperator{\Prod}{Prod}

\DeclareMathOperator{\Gen}{Gen}

\newcommand{\eat}[1]{}

\makeatletter
\DeclareRobustCommand*\cal{\@fontswitch\relax\mathcal}
\makeatother

\begin{document} 


\title[Pure projective tilting and Goresntein pure projective tilting modules]{Pure projective tilting modules associated with a special ring and Gorenstein properties}
\author{Umamaheswaran Arunachalam}
\address{SASTRA Deemed to be University, Thirumalaisamudram, Thanajvur
}
\address{Institut des Hautes Études Scientifiques,
   Université Paris-Saclay, 
   Le Bois-Marie,  
   35 route de Chartres  CS 40001, 
   91893 Bures-sur-Yvette, France.}
\email{umamaheswaran@maths.sastra.edu, arunachalam@ihes.fr}
\urladdr{https://sites.google.com/site/ruthreswaranuma/home}


\begin{abstract}
In this paper, we study pure-projective tilting modules and related classes of rings. We introduce the notion of a pure-tilting hereditary ring, namely, a ring over which every ideal is pure-projective tilting, and investigate its structural properties. We prove that $R$ is a pure-tilting hereditary ring if and only if $R$ is hereditary noetherian over a von Neumann regular ring $R$. In the commutative case, we show that $R$ is a pure $1$-tilting hereditary  ring precisely when $R$ is hereditary noetherian.  Using Kaplansky’s conjecture, we establish a connection between pure-tilting hereditary rings and the hereditary Noetherian property of prime factor rings. In the category theory, for the torsion pair $(\mathcal T, \mathcal F)=(\operatorname{Gen}(I), I^{\perp})$ in $R\text{-}\Mod$, we establish that the associated Happel--Reiten--Smal\o heart $\mathcal H_I$ is a Grothendieck category. We also examine the characterize $\Ext$-orthogonal classes determined by pure projective tilting modules. In addition, we show that every Gorenstein pure projective tilting module is Gorenstein flat if and only if every Gorenstein pure projective tilting module is strict $\mathscr{T}$-stationary, where $\mathscr{T}$ denotes the class of all finitely presented tilting modules. These results establish new links between tilting theory, hereditary ring conditions, and Gorenstein homological structures.
\end{abstract}

\maketitle

\section{Introduction}	
Throughout this paper, $R$ denotes an associative ring with identity, and all modules are assumed to be left $R$-modules unless explicitly stated otherwise. The category of left $R$-modules is denoted by $R$-$\Mod.$

A module is said to be pure-projective if it is projective with respect to pure-
exact sequences (See Preliminaries). According to Warfield’s criteria \cite[Corollary 3]{WAR:1969}, a module is pure-projective if and only if it is a direct summand of a direct sum of finitely presented modules. However, a pure-projective module need not itself decompose as a direct sum of finitely presented modules. This naturally leads to the question of identifying classes of rings over which such a decomposition property does hold. In \cite{PP:2005}, Puninski and Rothmaler investigated the broader problem of determining when every pure-projective module is a direct sum of finitely generated modules. In particular, they proved that if $R$ is a hereditary Noetherian ring, then every pure-projective left (respectively, right) $R$-module is a direct sum of finitely generated left (respectively, right) $R$-modules \cite[Corollary 6.5]{PP:2005}. Earlier, Warfield showed that over a commutative complete Noetherian local ring, every pure-projective module decomposes as a direct sum of finitely presented modules \cite[Corollary 4]{WAR:1969}. Motivated by these results, the purpose of this paper is to extend such decomposition theorems to broader classes of rings.

Pure-projective tilting modules were studied by Bazzoni et al. in \cite{BH}. In that work, they addressed the pure-projective version of a question raised by Saorín, namely: for which rings R is every pure-projective 1-tilting module classical? Their results are fundamental to our approach and will be used in the proof of Theorem~\ref{2.2}, one of the main results of this paper.

 Recall that a ring $R$ is hereditary whenever every ideal of $R$ is projective. In a similar spirit, one may call $R$ pure hereditary in \cite{MD:2017} provided every ideal of $R$ is pure-projective. Motivated by this analogy, we introduce the notion of a pure-tilting hereditary ring, namely, a ring for which every ideal is pure-projective tilting.

 By \cite[Exercise 4.22, p.~161]{LM:1999}, a commutative ring $R$ is coherent and perfect if and only if $R$ is Artinian. Moreover, Remark~\ref{3.8.R} provides an example showing that a \textit{pure-tilting hereditary ring} need not be semisimple because a commutative von Neumann regular ring which is not Artinian cannot be pure-tilting hereditary. This establishes the existence of a nontrivial class of pure-tilting hereditary rings properly containing the class of semisimple rings. Consequently, we obtain the following hierarchy of implications and structural relationships among these classes of rings. In the following diagram, all the arrows represent inclusions.
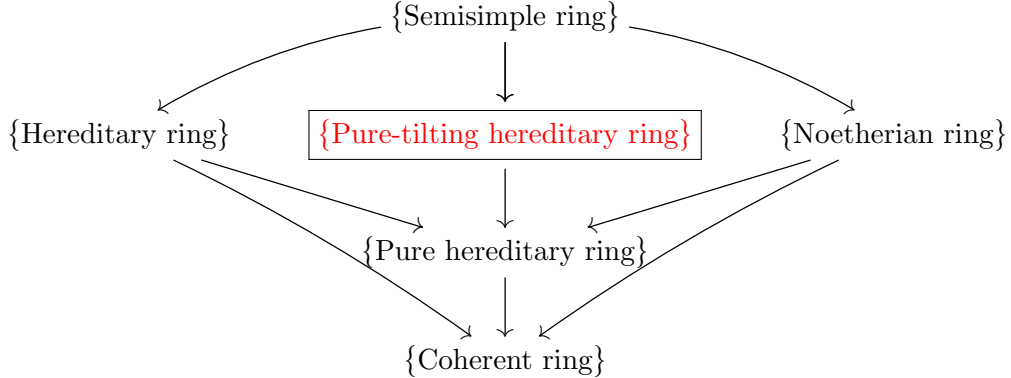
\begin{center}
\[\xymatrix@C-0.15pc@R-.18pc{
&  & {\{\text{\small Semisimple ring}\}} \ar[d] \ar@/{}^{-1pc}/[ld] \ar[d]\ar@/{}^{1pc}/[rd]& & &\\ 
  &  {\{\text{\small Hereditary ring}\}} \ar@/{}^{0.4pc}/[rdd]\ar[rd] & \boxed{\color{red} \{\text{\small Pure-tilting hereditary ring}\}} \ar@/{}^{-0.0pc}/[d]  & {\{\text{\small Noetherian ring}\}}\ar[ld]\ar@/{}^{-0.4pc}/[ldd]&&\\
&   & {\{\text{\small Pure hereditary ring}\}} \ar[d] & & &\\
 && {\{\text{\small Coherent ring}\}}& & 
}\]
\end{center}

 Note that every pure projective module is a direct summand of a direct sum of of finitely presented modules. Every finitely presented module is a pure projective tilting module but converse not necessarily true. It follows that all $\mathscr{T}$-injective modules are $\fp$-injective modules and converse not necessarily true. Over a coherent ring all finitely presented modules are equal to all pure projective tilting modules. So $\fp$-injective modules coincides with the $\mathscr{T}$-injective modules over a coherent ring.

\begin{question}
Is every pure-tilting hereditary ring hereditary noetherian?
\end{question}

We provide an affirmative answer in two important special cases. The first is the case of von Neumann regular rings in the non-commutative setting, where the conclusion follows from Proposition~\ref{2.3}. The second occurs when R is a commutative pure $1$-tilting hereditary ring.  These results indicate that the pure-tilting hereditary condition imposes strong structural constraints, often forcing hereditary noetherian behavior in a wide range of contexts. This observation motivates a more detailed investigation of the relationship between flat modules and pure projective modules, leading to the following main results.

\begin{customtheorem}{\ref{2.3}}\label{MAIN-1}
 Let $R$ be a von Neumann regular ring. Then $R$ is a  pure-tilting hereditary ring if and only if $R$ is hereditary noetherian. 
 \end{customtheorem}

To complement the previous result, we next consider the commutative setting. In this case, the structure of pure $1$-tilting hereditary rings is more rigid due to the stronger interplay between purity and finiteness conditions in commutative module categories.

\begin{customtheorem}{\ref{2.2}}
 Let $R$ be a commutative ring. Then $R$ is a  pure $1$-tilting hereditary ring if and only if $R$ is hereditary Noetherian. 
\end{customtheorem}

A central aspect in the study of torsion pairs arising from the tilting theory is the structure of the associated Happel--Reiten--Smal{\o} (RHS) heart. In particular, determining when this heart forms a Grothendieck category is important, since Grothendieck categories provide a natural framework for studying exactness properties and direct limits. In the present setting, this question is closely related to the structural properties of the induced torsion pair. The study of pure $1$-tilting hereditary rings naturally leads to the examination of the torsion theory determined by the corresponding pure projective tilting module. In particular, it is important to understand the definability properties of the torsion-free class, the finiteness conditions of the induced torsion pair, and the structure of the associated heart. We summarize these properties in the following theorem.

\begin{customtheorem}{\ref{GROTH}}
Let $R$ be a pure 1-tilting hereditary ring. Suppose $(\mathcal T,\mathcal F)=(\operatorname{Gen}(I),I^{\perp})$ is the induced torsion pair in the Grothendieck category $R\text{-}\Mod$. Then:
\begin{enumerate}
    \item the class  $I^{\perp} = \{M\in R\text{-Mod}\mid \operatorname{Ext}_R^1(I,M)=0\}$
    is a definable subcategory of $R\text{-}\Mod$;
    \item the torsion pair $(\mathcal T,\mathcal F)$ is of finite type;
    \item the associated Happel--Reiten--Smal\o{} heart $\mathcal H_I$ is a Grothendieck category.
\end{enumerate}
\end{customtheorem}

The following result establishes a characterization of Gorenstein flatness for Gorenstein pure projective tilting modules in terms of strict stationarity with respect to the class of finitely presented tilting modules. The theorem below shows that, within the class of Gorenstein pure projective tilting modules, the Gorenstein flat property can be completely detected through this stationarity condition. Thus, it gives a homological criterion that connects Gorenstein flatness with the structural properties arising from tilting theory.

\begin{customtheorem}{\ref{MAIN:GOREN}}
Let $\mathscr{T}$ be the class of all finitely presented tilting modules. Then every Gorenstein pure projective tilting module is Gorenstein flat if and only if every Gorenstein pure projective tilting module is a strict $\mathscr{T}$-stationary module. 
\end{customtheorem}

We conclude this section by providing a concise summary of the paper's concepts. Section 2 contains necessary notation and definitions required for the subsequent sections. 

Section 3 is devoted to the introduction and study of pure-tilting hereditary rings. In this section, we establish several equivalent characterizations of these rings and investigate their fundamental properties. We prove that a ring $R$ is pure hereditary if and only if $R$ is Noetherian over a $\CF$-ring. We further show that, for a von Neumann regular ring $R$, the ring $R$ is pure-tilting hereditary if and only if it is hereditary Noetherian. We also consider semiregular rings and prove that, for such a ring $R$, the conditions that $R$ is Noetherian and that $R/J(R)$ is pure-tilting hereditary are equivalent to $R$ being Artinian and perfect. In addition, we establish that a ring $R$ is semisimple if and only if it is an $\IF$-ring and hereditary Noetherian, provided $R$ is pure-tilting hereditary. 

In Section 4, we investigate pure $1$-tilting hereditary ring with the commutative setting and analyze the framework of category theory. We prove that a ring $R$ is pure $1$-tilting hereditary if and only if $R$ is semisimple. Finally, we show that a ring $R$ is a Dedekind domain if and only if it is a pure $0$-tilting hereditary ring if and only if it is a pure-tilting hereditary ring. In Category theory, we assume that $R$ is a pure $1$-tilting hereditary ring. Suppose $(\mathcal T,\mathcal F)=(\operatorname{Gen}(I),I^{\perp})$ is the induced torsion pair in the Grothendieck category $R\text{-}\Mod$. Then we prove the main Theorem \ref{GROTH} of the section. Moreover, we prove that the induced torsion pair $({\rm Gen}(I),I^{\perp})$ is a tilting torsion pair of finite type in the category $R\text{-Mod}$. Finally, for every $1$-tilting ideal $I$, the heart $\mathcal H_t$ of the $t$-structure in $\mathcal D(\Mod\text{-}R)$ induced by the torsion pair $(\Gen(I),\mathcal F)$ is a Grothendieck category.

 In Section 5, we investigate the $\Ext$-right orthogonal class of pure projective tilting modules over pure-tilting hereditary ring. Let $\mathcal{P}$ denote the class of all pure projective $R$-modules. We first prove that over a Noetherian ring, a module is injective if and only if $\mathcal{P}$-injective. We also show that over an arbitrary ring, every cotorsion module is $\mathscr{T}$-injective if and only if all pure projective tilting modules are projective, and all $R$-modules are $\mathscr{T}$-injective, equivalently, every $R$-module is $\mathscr{T}$-injective. Moreover, the class $\mathscr{T}^{\bot}$ of all $\mathscr{T}$-injective modules is closed under pure submodules. We further prove that every $R$-module has an $\mathscr{T}$-injective preenvelope over a  pure-tilting hereditary ring. The main result of this section establishes that the following conditions are true  over a  pure-tilting hereditary ring $R$. (i) Every factor module of an $\mathscr{T}$-injective module by an absolutely pure submodule is $\mathscr{T}$-injective, (ii) Every absolutely pure module is $\mathscr{T}$-injective, (iii) every absolutely pure module has injective dimension at most $1.$

In Section 6, we introduce and study the notion of Gorenstein pure-projective tilting modules. We establish the equivalent conditions characterizing the relationship between Gorenstein projective modules and Gorenstein pure-projective tilting modules over simple Artinian rings. The section culminates with the proof of Theorem~\ref{MAIN:GOREN}, which constitutes the main result of this section.

\section{Preliminaries}

In this section, we collect several known definitions and terminological conventions that will be used throughout the remainder of the paper.

Recall that an exact sequence $0 \rightarrow A \rightarrow B \rightarrow C \rightarrow 0$ of left $R$-modules is said to be pure exact if and only if for every finitely presented $R$-modules $M,$ the induced sequence $0 \rightarrow \Hom_R(M, A) \rightarrow \Hom_R(M, B) \rightarrow \Hom_R(M, c) \rightarrow 0$ is exact. A left $R$-module $M$ is said to be pure-projective if it is projective with respect to pure exact sequences.

 \begin{lemma}\label{2.1}
 Let $M$ be an $R$-module. Then the following conditions are equivalent:
 \begin{enumerate}
     \item [(1)] $M$ is pure projective;
     \item [(2)] Every pure exact sequence $0 \rightarrow A \rightarrow B \rightarrow M \rightarrow 0$ of $R$-modules splits;
     \item [(3)] $M$ is a direct summand of a direct sum of finitely presented $R$-modules.
 \end{enumerate}
 \end{lemma}

Given a class $\mathscr{C}$ of left $R$-modules, we write

\[
\mathscr{C}^{\perp}
=
\left\{
N \in R\text{-}\mathrm{Mod}
\;\middle|\;
\Ext_R^{i}(M,N)=0
\text{ for all } M\in \mathscr{C}
\right\},
\]
\[
{}^{\perp}\mathscr{C}
=
\left\{
N \in R\text{-}\mathrm{Mod}
\;\middle|\;
\Ext_R^{i}(N,M)=0
\text{ for all } M\in \mathscr{C}
\right\}.
\]

%

Let $\mathscr{C}$ be a class of left $R$-modules. Following \cite{Eno1}, 
we say that a map $f \in \Hom_{R}(C, M)\mbox{ with }C \in \mathscr{C}$ is a $\mathscr{C}$-precover of $M$, if the group homomorphism \[\Hom_{R}(C^\prime, f) \colon \Hom_{R}(C^\prime, C) \rightarrow \Hom_{R}(C^\prime, M)\] is surjective for each $C^\prime$ $\in \mathscr{C}$. A $\mathscr{C}$-precover $f \in \Hom_{R}(C, M)$ of $M$ is called a $\mathscr{C}$-cover of $M$ if $f$ is right minimal. That is, if $fg = f$ implies that $g$ is an automorphism for each $g \in$ $\End_{R}(C)$. $\mathscr{C} \subseteq$ $R\mbox{-}\Mod$ is a precovering class (covering class) provided that each module has a $\mathscr{C}$-precover ($\mathscr{C}$-cover). Dually, we have the definition of $\mathscr{C}$ preenvelope ($\mathscr{C}$ envelope).

Let $\mathcal{A}, \mathcal{B} \subseteq R$-$\Mod$. The pair $(\mathcal{A}, \mathcal{B})$ is called a \textit{cotorsion pair} \cite{RGT:2006}, if $\mathcal{A} =\,^\perp\mathcal{B}$ and $\mathcal{A}^\perp =\mathcal{B}$

An $R$-module $T$ is called a $n$-tilting \cite{RGT:2006} provided that 
    \begin{enumerate}
        \item $\pd_R(T) \leq n$
        \item $\Ext_R^i(T, T^(\kappa)) = 0$ for all $1 \leq i \leq \omega$ and for all cardinals $\kappa.$
        \item There are $r \geq 0$ and a long exact sequence $0 \rightarrow R \rightarrow T_0 \rightarrow \cdots \rightarrow T_r \rightarrow 0,$ where $T_i \in \Add (T) $ for all $i \leq r.$
    \end{enumerate}
    
 A \textit{classical tilting} module is a tilting module having a projective resolution by finitely generated projectives. 

 Let $\mathcal{C}$ be a class of modules. Then $\mathcal{C}$ is \textit{definable} if $\mathcal{C}$ is closed under direct limits, direct products, and pure submodules.

 Recall that a ring $R$ is said to be left hereditary (resp. left semihereditary) if every left ideal (resp. finitely generated left ideal) of $R$ is projective. A ring $R$ is semiperfect if every finitely generated left (or right) $R$-module has a projective cover or, equivalently, if $R/J(R)$ is semisimple and the idempotents of $R/J(R)$ can be lifted to $R$. Also, a ring $R$ is called left perfect if every left $R$-module has a projective cover, or, equivalently, if $R/J(R)$ is semisimple and $J(R)$ is left $T$-nilpotent.

 The category $\mathscr{A}$ is \textit{cocomplete} \cite{BOS:1975} if $\lim\limits_{\longrightarrow}F$ \textit{(colimit or inductive limit)} exists for every $F \colon \mathscr{I} \rightarrow \mathscr{A},$ where $\mathscr{I}$ is a small category. A cocomplete abelian category $\mathscr{G}$ is called a \textit{Grothendieck category} if direct limits are exact in $\mathscr{G}$ and $\mathscr{G}$ has a generator. 

 A module is \textit{Gorenstein projective} if it occurs as a syzygy in an acyclic complex 
 \begin{equation}
   \mathrm{P}_\bullet \colon \cdots \rightarrow P_1 \rightarrow P_0 \rightarrow P^0 \rightarrow P^1 \cdots 
\end{equation}
 of projective modules that remains exact after applying the functor $\Hom_R(-, P)$ for every projective $R$-module $P.$ Similarly, an $R$-module is called \textit{Gorenstein flat} if it occurs as a syzygy in an acyclic complex $(2.1)$ of flat $R$-modules that remains exact after tensoring with an arbitrary injective right $R$-module.

\begin{proposition} \cite[Proposition ~1.5]{ET:2011} \label{P1}
Let $M,N$ be two left $R$-modules and fix a non-negative integer $n$. Then the following conditions are equivalent.
\begin{enumerate}
    \item[(i)] The natural map
    \[
    \Phi_M^{(n)} :
    \operatorname{Tor}^{R}_{n}(\operatorname{Hom}(N,D),M)
    \longrightarrow
    \operatorname{Hom}(\operatorname{Ext}^{n}_{R}(M,N),D)
    \]
    is injective for any divisible adelian group $D$.

    \item[(ii)] The $n$th syzygy module of $M$ is strict Mittag--Leffler over $N$.
\end{enumerate}
\end{proposition}

The following lemma follows from 

\begin{lemma} \cite[Proposition ~1.10]{LH:2008} \label{P2}
    Let $A$ be a left $R$-module and $B$ a right $R$-module. Then \[\Tor^R_n((B)^\kappa, A) \longrightarrow \Tor^R_n(B, A)^\kappa,\] is injective for any cardinal $\kappa$ if and only if the $n^{th}$-syzygy $\Omega^n(M)$ of $M$ is $B$-Mittag-Leffler. 
\end{lemma}

The following lemma follows from \cite[Proposition ~8.14]{LH:2008}

 \begin{lemma} \label{P3}
     Let $M$ and $B$ be left $R$-modules. Then $M$ is $\Hom_{\mathbb{Z}}(B, \mathbb{Q}/\mathbb{Z})$-Mittag-Leffler if $M$ is strict $B$-stationary.
 \end{lemma}

 The fundamental concepts of quivers and representation theory relevant to our examples can be found in \cite{ASS:2010}.

\section{Pure projective tilting modules over Pure-tilting hereditary rings}

In this section, we introduce the notion of a pure-tilting hereditary ring and investigate several equivalent characterizations of this class of rings. We establish a number of fundamental results concerning their structural and homological properties. In particular, we examine the relationship between semisimple modules and pure tilting hereditary rings.

\begin{proposition}\label{2.5}
Let $M$ be an $R$-module. If $M$ is finitely generated, $M$ is finitely presented if and only if $M$ is pure projective.
\end{proposition} 
\begin{proof}
The direct implication is clear because every finitely presented module is pure projective. Conversely, let $M$ be a pure projective $R$-module. For any absolutely pure module $A,$ there is an exact sequence $0 \rightarrow A \rightarrow E \rightarrow E/A \rightarrow 0,$ which inducs an exact sequence $\Hom_R(M, E) \rightarrow \Hom_R(M, E/M) \rightarrow \Ext_R^1(M, A) \rightarrow 0.$ By hypothesis, $\Ext_R^1(M, A) = 0.$ The rest of the proof will be followed by \cite[Proposition]{Eno1}.
\end{proof}

An $R$-module $M$ is called \textit{pure projective tilting} if it is both pure projective and tilting. Any ring over itself is a pure projective $0$-tilting module. In general, the class of all finitely presented tilting modules is pure projective tilting, but the converse need not be true. However, the converse does hold for classical tilting modules. Consequently, the class of all classical pure projective tilting modules and the class of all finitely presented tilting modules coincide by Proposition \ref{2.5}. 
 
 \begin{definition}
 A ring $R$ is called pure-tilting hereditary if every ideal is pure projective tilting.
 \end{definition}

 \begin{remark}
 Every  pure-tilting hereditary ring is pure hereditary. Assume that every tilting $R$-module is finitely presented. Then every  pure-tilting hereditary ring is Noetherian. Consequently, if the class of tilting modules coincides with the class of finitely presented modules, then $R$ is a  pure-tilting hereditary ring if and only if $R$ is Noetherian. In particular, this condition is satisfied over a von Neumann regular rings by \cite[Corollary 6.2.4]{RGT:2006}.
 \end{remark}

 We will discuss with a basic quiver theoretic example illustrating the definition of pure-tilting hereditary ring. The following example shows that the class of pure $0$-tilting hereditary rings naturally includes certain path algebras arising from quiver representations. In particular, the path algebra associated with the Jordan quiver provides a standard nontrivial example.

 \begin{example}
     Consider the quiver $Q:\quad \bullet \xrightarrow{;\alpha;} \bullet.$ with a single vertex and one loop $(\alpha)$. The corresponding path algebra is $kQ \cong k[x]$. Since $k[x]$ is a principal ideal domain, it is a Dedekind domain. Hence, every nonzero ideal of $k[x]$ is invertible, and therefore every nonzero ideal is a progenerator. Recall that $0$-tilting modules are precisely the progenerators. Consequently, every nonzero ideal of $k[x]$ is a $0$-tilting module. Moreover, every ideal of $k[x]$ is finitely generated and projective, and hence pure projective. Therefore, every nonzero ideal of $k[x]$ is a pure-projective $0$-tilting module. Thus, the path algebra of the Jordan quiver provides a standard nontrivial quiver-theoretic example of a pure $0$-tilting hereditary ring.
 \end{example}

\begin{example}
Every Noetherian von Neumann regular ring R is a  pure-tilting hereditary ring. In particular, every semisimple ring is pure-tilting hereditary. However, the converse does not hold in general.
\end{example}

\begin{proof}
By \cite[Corollary 6.2.4]{RGT:2006}, over a von Neumann regular ring the class of all tilting modules coincides with the class of all projective generators. Since $R$ is Noetherian, every ideal of $R$ is finitely generated, and hence projective. Therefore, every ideal is a pure projective tilting module. It follows that $R$ is a  pure-tilting hereditary ring. In particular, every semisimple ring is von Neumann regular and Noetherian, and therefore it is pure-tilting hereditary. Conversely, assume that $R$ is a  pure-tilting hereditary ring. Then every ideal of $R$ is a pure projective tilting module. However, pure projective tilting modules need not be projective. Consequently, ideals of $R$ are not necessarily projective, and hence not necessarily flat. Therefore, $R$ need not be a von Neumann regular ring. Thus, the converse implication does not hold.
\end{proof}

We observe that every semisimple ring is an example of a pure-tilting hereditary ring. However, by the above example and Remark (\ref{3.8.R}), there exist pure-tilting hereditary rings which are not necessarily semisimple. Thus, the class of  pure-tilting hereditary rings properly contains the class of semisimple rings.

  \begin{proposition}\label{2.0.2}
Every pure hereditary ring is coherent.
 \end{proposition}
 
 \begin{proof}
 Every finitely generated ideal is evidently pure projective, and by Proposition \ref{2.5}, it is also finitely presented.
 \end{proof}
 
\begin{remark}\label{REM:3.8}
Every  pure-tilting hereditary ring is coherent. This follows immediately from the fact that every  pure-tilting hereditary ring is pure hereditary, together with Proposition~\ref{2.0.2}.
\end{remark}

We have the following
\begin{proposition} \label{2.4.0}
    Every noetherian ring $R$ is pure hereditary.
\end{proposition}

We now turn to the natural question: When does the converse hold? Before addressing this, recall that a ring $R$ is $CF$ if every cyclic $R$-module can be embedded in a free $R$-module. We prove the following theorem over $CF$-ring 
\begin{theorem}
 Let $R$ be a $CF$-ring. Then $R$ is a pure hereditary ring if and only if $R$ is noetherian. 
 \end{theorem}
 
 \begin{proof}
  Assume that $R$ is a pure hereditary  ring. Then every ideal is pure projective. In particular, every cyclic module is pure projective. It follows that $R/I$ is pure projective. Let $A$ be an absolutely pure module. Then the following sequence $0 \rightarrow A \rightarrow E \rightarrow E/A \rightarrow 0$ is pure exact. Thus, $\Ext_R^1(R/I, A) = 0.$ Hence, every absolutely pure module is injective. Then the theorem follows from \cite[Theorem ~3]{CM}. The converse of the assertion is true from the Proposition ~\ref{2.4.0}.
 \end{proof}

 \begin{corollary}
    Let $R$ be a $CF$-ring. Then every  pure-tilting hereditary ring is Noetherian.
 \end{corollary}

Kaplansky formulated the following conjecture in \cite{KAP:1970}, which was subsequently proved by J. W. Fisher and R. L. Snider in \cite{FS:1974}. 
 
\begin{conjecture*} [(Kaplansky Conjecture:)]
    A ring $R$ is von Neuman regular if and only if $R$ is semiprime and each prime factor ring of $R$ is von Neumann regular. 
\end{conjecture*}

Motivated by this result, in the present work we establish a connection between pure-tilting hereditary rings and the hereditary Noetherian property of prime factor rings. As a first step toward proving this relation, we prove the following result.

 \begin{theorem}\label{2.3}
 Let $R$ be a von Neumann regular ring. Then $R$ is a  pure-tilting hereditary ring if and only if $R$ is hereditary noetherian. 
 \end{theorem}
 
 \begin{proof}
 Assume that \(R\) is a  pure-tilting hereditary ring. Let \(I\) be an ideal of \(R\). By hypothesis, \(I\) is a pure projective tilting module. Consider a short exact sequence $\eta \colon 0 \rightarrow K \rightarrow F \rightarrow I \rightarrow 0,$ where $F$ is free. Since \(I\) is flat, the sequence \(\eta\) is pure exact. Hence, by Lemma~\ref{2.1}, the sequence splits. Therefore, \(I\) is projective. It follows that every ideal of \(R\) is projective, and thus \(R\) is hereditary. Moreover, \(R\) is Noetherian because the class of tilting modules coincides with the class of projective generators, all of which are finitely presented. Conversely, suppose that \(R\) is hereditary Noetherian. Then every ideal of \(R\) is finitely presented, and hence pure projective. Since $R$ is von Neumann regular, then every ideal is pure projective tilting.  as required. Therefore, \(R\) is a  pure-tilting hereditary ring.
 \end{proof}

The following corollary gives a connection between pure-tilting hereditary rings and the hereditary Noetherian property of prime factor rings.

\begin{corollary}
    Let $R$ be a von Neuman regular ring.  Then $R$ is a pure-tilting hereditary ring if and only if each factor ring of $R$ is semiprime and each prime factor ring is hereditary noetherian. 
\end{corollary}

  The $\F$-semiperfect ring introduced by Oberst and Schneider \cite{OS:1971}. A ring $R$ is said to be semiregular if $R/J(R)$ is von Neumann regular and idempotents lift modulo $J(R)$ in \cite{LA}.

 \begin{theorem}\label{3.3.3}
Let \(R\) be a semiregular ring. Then \(R\) is Noetherian and \(R/J(R)\) is a  pure-tilting hereditary ring if and only if \(R\) is Artinian and perfect.
\end{theorem}

\begin{proof}
Assume that \(R\) is Noetherian and that \(R/J(R)\) is a  pure-tilting hereditary ring. Since \(R\) is semiregular, the quotient ring \(R/J(R)\) is von Neumann regular. Hence, by Proposition~\ref{2.3}, \(R/J(R)\) is hereditary Noetherian. A von Neumann regular hereditary Noetherian ring is semisimple Artinian. Therefore, $R/J(R)$ is semisimple Artinian. Since \(R\) is semiregular, \(J(R)\) is Jacobson radical and idempotents lift modulo \(J(R)\). Consequently, \(R\) is semiperfect. Moreover, because \(R\) is Noetherian and \(R/J(R)\) is semisimple Artinian, it follows that \(J(R)\) is nilpotent. Hence \(R\) is Artinian. In particular, every Artinian ring is perfect, and therefore \(R\) is perfect. Conversely, suppose that \(R\) is Artinian and perfect. Then \(R/J(R)\) is semisimple Artinian. Every semisimple Artinian ring is hereditary Noetherian and von Neumann regular. Therefore, by Proposition~\ref{2.3}, \(R/J(R)\) is a  pure-tilting hereditary ring. Since every Artinian ring is Noetherian, the conclusion follows.
\end{proof}

A standard example of a commutative von Neumann regular ring which is not Artinian is the infinite direct product of fields $R=\prod_{i\in I} F_i,$ where $I$ is an infinite index set. It is well known that arbitrary direct products of fields are von Neumann regular, while infinite products fail to be Artinian. Hence, $R=\prod_{i\in I} F_i,$ is not perfect because \cite[Exercise 4.22, Page 161]{LM:1999}. By the previous theorem, \(R/J(R)\) is an example for a von Neumann regular ring but it is not pure-tilting hereditary. This example naturally leads to the following remark.

\begin{remark} \label{3.8.R}
A commutative von Neumann regular ring which is not Artinian cannot be pure-tilting hereditary. 
\end{remark}

\begin{proof}
    By \cite[Exercise 4.22, Page 161]{LM:1999}, $R$ is not perfect. Therefore, $R$ cannot be a pure-tilting hereditary ring by Proposition \ref{3.3.3}.
\end{proof}

  In the following proposition we will give an equivalent for von Neumann regular ring with $IF$-ring. These rings have been studied by Colby. We recall that a ring R is called $IF$ if every injective left $R$-module is flat.
 
 \begin{proposition}\label{2.4}
 Let $R$ be a  pure-tilting hereditary ring. Then $R$ is von Neumann regular if and only if $R$ is $IF$ and hereditary. 
 \end{proposition}
 
 \begin{proof}
  The direct implication follows by Proposition \ref{2.3} and every $R$-module is flat. Conversely, we assume that $R$ is an $IF$ and hereditary noetherian ring. We need to show that for any $R$-module $M$ is an absolutely pure. Consider an exact sequence $\epsilon \colon 0 \rightarrow M \rightarrow E(M) \rightarrow E(M)/M \rightarrow 0,$ where $E(M)$ is an injective envelope of $M.$ The sequence $\epsilon$ is pure exact since by hypothesis $E(M)/M$ is flat. It follows that $M$ is an absolutely pure. Thus, $R$ is a von Neumann regular ring.
 \end{proof}
 
 In Proposition \ref{2.4}, we get $R$ is noetherian in the direct implication of Proposition \ref{2.3}. Then we have the following
 \begin{corollary}
 Let $R$ be a  pure-tilting hereditary ring. Then $R$ is semisimple if and only if $R$ is $IF$ and hereditary noetherian. 
 \end{corollary}
 
\section{Pure 1-tilting hereditary ring}

In this section, we investigate pure $1$-tilting hereditary rings and explore their structural relationships with semisimple rings and Dedekind domains. Moreover, within the setting of category theory, we study their correspondence with Grothendieck categories and establish several important results.

We begin our discussion with the following observations.

\begin{lemma} \cite[Corollary ~6.5]{PP:2005} \label{LMA:3.1}
    Every pure projective module over a hereditary noetherian ring is a direct sum of finitely generated modules
\end{lemma}
 Note that a $1$-tilting module equivalent to a finitely presented $1$-tilting module is called classical. Bazzoni et al. in \cite[Corollary ~2.8]{BH} proved that every pure projective $1$-tilting module is classical if $R$ is a ring over which every pure projective module is a direct sum of finitely presented modules. Then $T$ is classical because $R$ is the right hereditary right Noetherian by Lemma \ref{LMA:3.1} and Proposition \ref{2.5}.

\begin{example}
Let $k$ be a field, and consider the Kronecker quiver and let 
$Q:\quad
1 \overset{\alpha,\beta}{\rightrightarrows} 2$
and let $R=kQ$ be the corresponding path algebra. Since $Q$ is acyclic, $R$ is a hereditary ring. Moreover, $R$ is a finite-dimensional $k$-algebra, and hence it is left and right Noetherian. Let $P(1)$ and $P(2)$ denote the indecomposable projective $R$-modules corresponding to the vertices $1$ and $2$, respectively, and let $S(2)$ be the simple module associated to the vertex $2$. Consider the module $T=P(1)\oplus P(2)\oplus S(2).$ Then $T$ is a pure projective tilting module over the hereditary Noetherian ring $R$. Since $R$ is hereditary, we have $\pd_R(T)\leq 1.$ Further, $\Ext_R^1(T, T) = 0.$ Indeed, the projective summands have trivial first extension groups, and $S(2)$ is injective in this orientation of the Kronecker quiver. Moreover, there exists an exact sequence $0\longrightarrow R\longrightarrow T_0\longrightarrow T_1\longrightarrow 0$ with $T_0,T_1\in \Add(T)$. Hence $T$ is a $1$-tilting module. Finally, since every finitely generated module over an Artin algebra is pure projective and $T$ is finitely generated, it follows that $T$ is a pure projective tilting module over the hereditary Noetherian ring $R$.
\end{example}

The corollary \cite[Corollary ~2.8]{BH} motivates us to prove the following
 
 \begin{proposition}
 Let $R$ be a pure $1$-tilting hereditary  ring over which every pure projective module is a direct sum of finitely presented modules, then every pure projective $1$-tilting module is classical and $R$ is semihereditary if and only if $R$ is hereditary noetherian.
 \end{proposition}
 
 \begin{proof}
 By hypothesis, every ideal is pure projective $1$-tilting. We assume that $R$ is a semihereditary ring and every pure projective $1$-tilting module is classical. It follows that every ideal is finitely presented (hence pure projective) flat module. Thus, $R$ is hereditary noetherian. The converse part is clear by \cite[Corollary ~2.8]{BH} and \cite[Corollary ~6.5]{PP:2005}.
 \end{proof}
 
  Note that a commutative ring $R$ is coherent if and only if every $R$-module has a flat preenvelope \cite[Theorem ~2.5.1]{xu}. Then we have the following
 
 \begin{theorem}\label{2.2}
 Let $R$ be a commutative ring. Then $R$ is a  pure $1$-tilting hereditary ring if and only if $R$ is hereditary noetherian. 
 \end{theorem}
 
 \begin{proof}
 We assume that $R$ is a pure $1$-tilting hereditary ring. By \cite[Theorem ~3.7]{BH}, every pure projective $1$-tilting module is projective. By hypothesis, every ideal is projective. Hence $R$ is a hereditary ring. Also, $R$ is noetherian since every ideal is finitely presented by \cite[Theorem ~3.7]{BH}. Conversely, if $R$ is a hereditary noetherian ring, then $w.gl. \dim (R) = gl.\dim (R) \leq \infty.$ By \cite[Theorem ~2.5.1, Theorem ~4.5.3]{xu} and Proposition \ref{2.0.2}, $R$ is a von Neumann regular ring. Then the class of all tilting modules coincide with the class of all projective generators. Thus, every ideal is pure projective tilting.  The semihereditary ring coincide with the hereditary ring over a noetherian ring. It follows that every tilting module has projective dimension is $\leq 1.$  Hence every ideal is pure projective $1$-tilting.
 \end{proof}

 According to Theorem \ref{2.2}, every  pure-tilting hereditary ring is necessarily pure $1$-tilting hereditary on a commutative ring. Moreover, a semisimple ring is exactly a Noetherian von Neumann regular ring. These facts lead to the following result.
 
 \begin{corollary}\label{2.0.3}
 Let $R$ be a commutative ring. Then a ring $R$ is pure $1$-tilting hereditary  if and only if $R$ is semisimple.  
 \end{corollary}
 
 \begin{proposition}
 Let $R$ be a commutative ring. Then the following are equivalent:
 \begin{enumerate}
     \item $R$ is a pure $1$-tilting hereditary ring. 
     \item $R$ is a hereditary noetherian ring.
     \item $R$ is a hereditary Artinian ring.
     \item $R$ is a semisimple ring. 
 \end{enumerate}
 \end{proposition}
 
 \begin{proof}
  $(1) \Leftrightarrow (2)$ follows from Theorem \ref{2.2},  $(1) \Leftrightarrow (4)$ follows from Corollary \ref{2.0.3} and $(3) \Leftrightarrow (4)$ follows by \cite[Propositoin ~2.14]{MD}.
 \end{proof}

 \begin{lemma}\label{3.17}
Let $R$ be a commutative domain. Then $R$ is a Dedekind domain if and only if $R$ is a pure $0$-tilting hereditary ring.
\end{lemma}

\begin{proof}
Recall that a module is $0$-tilting if and only if it is a progenerator. Suppose that $R$ is a Dedekind domain. Then every nonzero ideal of $R$ is invertible. In particular, every nonzero ideal is finitely generated and projective, and satisfies $IJ = R$ for some ideal $J$. Hence, every nonzero ideal is a generator of $R$-Mod. Therefore, every nonzero ideal is a progenerator and thus a $0$-tilting module. Since every ideal of a Dedekind domain is projective, it is also pure-projective. Hence $R$ is a pure $0$-tilting hereditary ring. Conversely, assume that $R$ is a pure $0$-tilting hereditary ring, i.e., every nonzero ideal of $R$ is a pure-projective $0$-tilting module. In particular, every nonzero ideal is a progenerator. Hence, every nonzero ideal is finitely generated, projective, and a generator. In particular, every nonzero ideal is invertible. It is a standard characterization that a commutative domain in which every nonzero ideal is invertible is a Dedekind domain. Hence, $R$ is a Dedekind domain. This completes the proof.
\end{proof}

In the previous lemma, we established equivalent conditions for a hereditary ring to be pure $0$-tilting, as well as for a ring to be a Dedekind domain. We now turn our attention to the general case of pure-tilting hereditary rings.

\begin{proposition}\label{3.18}
Let $R$ be a commutative domain. Then $R$ is a pure-tilting hereditary ring if and only if $R$ is a Dedekind domain.
\end{proposition}

\begin{proof}
Assume first that $R$ is a  pure-tilting hereditary ring. By Theorem~\ref{2.2}, $R$ is hereditary and noetherian. Since $R$ is commutative and hereditary, every ideal of $R$ is projective; hence $R$ is a Pr\"ufer domain. Moreover, hereditary rings have a weak global dimension at most one, so $\operatorname{w.gl.dim}(R)\leq 1.$
Because $R$ is also noetherian, it follows that $\operatorname{Kdim}(R)\leq 1.$ Therefore, $R$ is a Noetherian Pr\"ufer domain of Krull dimension at most one. Equivalently, $R$ is a Dedekind domain. Conversely, assume that $R$ is a Dedekind domain. Since every ideal of a Dedekind domain is projective, $R$ is hereditary. Moreover, Dedekind domains are noetherian domains. Hence, $R$ is hereditary noetherian. Applying Theorem~\ref{2.2}, we conclude that $R$ is pure-tilting hereditary.
\end{proof}

Combining Lemma~\ref{3.17} with Proposition~\ref{3.18}, we obtain the following characterization of the Dedekind domains.

\begin{corollary}
    Let $R$ be a commutative domain. Then the following conditions are equivalent.
\begin{enumerate}
    \item $R$ is a Dedekind domain;
    \item $R$ is a pure-tilting hereditary ring;
    \item $R$ is a pure $0$-tilting hereditary ring.
\end{enumerate}
\end{corollary}

\title{Grothendieck Category}

\subsection{Relation with Grothendieck Category}

A central question is understanding how the structural property of pure-tilting hereditary rings affect the torsion pair generated by the tilting module and, consequently, the nature of the associated Happel–Reiten–Smalø heart $\mathcal H_I,$ which arises as the heart of the $t$-structure determined by the torsion pair in the derived category $D(R).$

The following theorem shows that for a pure $1$-tilting hereditary ring, the tilting-induced torsion-free class is definable, the induced torsion pair is controlled by finitely presented objects (that is, it is of finite type), and these properties guarantee that the corresponding heart inherits the well-behaved structure of a Grothendieck category.

\begin{theorem} \label{GROTH}
Let $R$ be a pure 1-tilting hereditary ring. Suppose $(\mathcal T,\mathcal F)=(\operatorname{Gen}(I),I^{\perp})$ is the induced torsion pair in the Grothendieck category $R\text{-}\Mod$. Then:
\begin{enumerate}
    \item the class  $I^{\perp} = \{M\in R\text{-Mod}\mid \operatorname{Ext}_R^1(I,M)=0\}$
    is a definable subcategory of $R\text{-}\Mod$;
    \item the torsion pair $(\mathcal T,\mathcal F)$ is of finite type;
    \item the associated Happel--Reiten--Smal\o{} heart $\mathcal H_I$ is a Grothendieck category.
\end{enumerate}
\end{theorem}

\begin{proof}
Since $I$ is a $1$-tilting module, the pair $(\mathcal T,\mathcal F) = (\operatorname{Gen}(I),I^{\perp})$ is a torsion pair in $R\text{-Mod}$. Moreover, every tilting class is closed under extensions, products, direct sums, and epimorphic images. Because $I$ is assumed to be pure projective, it follows from standard results in tilting theory that the tilting class $I^{\perp} = \ker \operatorname{Ext}_R^1(I,-)$ is definable. Hence $I^{\perp}$ is closed under direct products, direct limits, and pure submodules. Since definable classes are closed under direct limits, the torsion-free class $\mathcal F=I^{\perp}$ is closed under direct limits. Therefore the torsion pair $(\mathcal T,\mathcal F)$ is of finite type. Now consider the Happel--Reiten--Smal\o{} heart $\mathcal H_I$ associated with the torsion pair $(\mathcal T,\mathcal F)$. By the theorem of Colpi, Gregorio, and Mantese on hearts of torsion pairs induced by tilting modules, the heart $\mathcal H_I$ is a Grothendieck category whenever the torsion-free class $\mathcal F$ is closed under direct limits. Since $\mathcal F=I^{\perp}$ is definable, it is closed under direct limits, and consequently $\mathcal H_I$ is a Grothendieck category.
\end{proof}

\begin{corollary}
Let $R$ be a pure $1$-tilting hereditary ring. Then the induced torsion pair $({\rm Gen}(I),I^{\perp})$ is a tilting torsion pair of finite type in the category $R\text{-Mod}$.
\end{corollary}

\begin{corollary}
Let $R$ be pure $1$-tilting hereditary ring. Then for every $1$-tilting ideal $I$, the heart $\mathcal H_t$ of the $t$-structure in $\mathcal D(\Mod\text{-}R)$ induced by the torsion pair $(\Gen(I),\mathcal F)$ is a Grothendieck category.
\end{corollary}

\begin{remark}
 Every ideal induces a definable torsion pair of finite type, and each corresponding HRS heart is a Grothendieck category. Consequently, the module category $R\text{-Mod}$ admits a highly controlled purity structure over a pure $1$-tilting hereditary ring $R$.
\end{remark}
 
\section{Ext-right orthogonal class of pure projective tilting Modules}

In this section, we prove the results by using homological properties. We begin with the following
 \begin{definition}
 Let $\mathscr{T}$ be the class of all pure projective tilting modules. A left $R$-module $M$ is called a $\mathscr{T}$-injective if $\Ext_R^1(T, M) = 0$ for all $T \in \mathscr{T}.$ We denote $\mathscr{T}^{\bot}$ by the class of all $\mathscr{T}$-injective modules.
 \end{definition}

 \begin{remark}
     Every injective $R$-module is $\mathscr{T}$-injective, but the converse does not hold. Furthermore, if $\mathscr{T}$ denotes a classical tilting class, then every $\fp$-injective $R$-module is also $\mathscr{T}$-injective, although the converse implication is not valid, in general. An $R$-module $G$ is $\mathscr{T}$-injective if and only if $G$ has injective property for each exact sequence $0 \rightarrow K \rightarrow M \rightarrow T \rightarrow 0$ of $R$-modules, where $T$ is pure projective tilting.
 \end{remark}

\begin{example}
Let \(R\) be an Artin algebra and let $\mathscr{T}$ be the class of all pure projective tilting $R$-modules. Since every finitely generated \(R\)-module over an Artin algebra is pure projective, every finitely generated tilting module belongs to \(\mathscr{T}\). For a tilting module \(T\in\mathscr{T}\), the associated tilting class is
$T^{\perp}
=
\{\, M\in \mathrm{Mod}\text{-}R
\mid
\operatorname{Ext}^1_R(T,M)=0
\,\}.
$
By the Brenner--Butler tilting theorem, $T^{\perp}=\operatorname{Gen}(T).$ Hence, every module generated by \(T\) is {\(T\)}-injective. In particular, any module belonging to
$\bigcap_{T\in\mathscr{T}} T^{\perp}$ is \(\mathscr{T}\)-injective.
\end{example}

We now illustrate the above concepts through an example arising from the representation theory of quivers and path algebras.
\begin{example}
Let \(R=kQ\) be the path algebra of the quiver $Q:\quad 1 \longrightarrow 2,$ where \(k\) is a field, and let 
$\mathscr{T}
=
\{\text{all pure projective tilting }R\text{-modules}\}.
$
Consider the tilting module $T=P(1)\oplus S(1),$ where \(P(1)\) denotes the indecomposable projective module corresponding to the vertex \(1\), and \(S(1)\) is the simple module at vertex \(1\). Since \(R\) is a finite-dimensional hereditary algebra, every finitely generated module is pure projective; hence \(T\in\mathscr{T}\). By tilting theory, $T^{\perp}=\operatorname{Gen}(T).$ Therefore, every module generated by \(T\) is \(T\)-injective. In particular, any module belonging to $\bigcap_{T\in\mathscr{T}} T^{\perp}$ is \(\mathscr{T}\)-injective.
\end{example}

 \begin{example}
 Let $(R, \mathfrak{m})$ be a local commutative ring. By \cite[Theorem ~3.7]{BH}, every pure projective $1$-tilting $R$-module is projective. Then the residue field $k = R/\mathfrak{m}$ is $\mathscr{T}$-injective since \cite[Proposition ~3.4]{BH}. Consider an exact sequence $0 \rightarrow k \rightarrow E(k) \rightarrow E(k)/k \rightarrow 0,$ where $E(k)$ is an injective envelope of $k.$ Then $k \oplus E(k)$ is $\mathscr{T}$-injective. 
 \end{example}
 
 \begin{proposition}\label{PRO:5.1}
 \begin{enumerate}
     \item Every injective $R$-module is $\mathcal{P}$-injective where $\mathcal{P}$ is the class of all pure projective $R$-modules. Converse part is hold over a noetherian ring.
     
     \item The class of all $\mathscr{T}$-injective modules coincide with the class of all injective modules over a noetherian ring and which the class of all pure projective tilting modules coincide with the class of all finitely generated modules.
 \end{enumerate}
 
 \end{proposition}
 \begin{proof}
  (1). Let $R$ be a noetherian ring. Then every cyclic module is pure projective. We show that every $\mathcal{P}$-injective module is injective. Let $G$ be an $\mathcal{P}$-injective module. Consider an exact sequence $0 \rightarrow G \rightarrow E(G) \rightarrow E(G)/G \rightarrow 0,$ where $E(G)$ is an injective envelope of $G.$ This sequence is pure exact since all finitely presented modules are pure projective and $G$ is $\mathcal{P}$-injective. It follows that we get an exact sequence $\Hom_R(R/I, E(G)) \rightarrow \Hom_R(R/I, E(G)/G) \rightarrow \Ext_R^1(R/I, G) \rightarrow 0$ when applying the functor $\Hom_R(R/I, -).$ Since $\Hom_R(R/I, E(G)) \rightarrow \Hom_R(R/I, E(G)/G)$ is surjective, $\Ext_R^1(R/I, G) = 0$. Hence $G$ is injective, as required.\\
  
  (2). By Proposition \ref{2.5}, the class of all pure projective tilting modules coincide with the class of all finitely presented modules. Hence we prove the remains with similar from the part $(1)$.
 \end{proof}

 Bazzoni et al. in \cite[Theorem ~3.7]{BH} proved that every pure projective $1$-tilting module is projective over a commutative ring. This result motivates us to prove the question of when all pure projective tilting modules are projective in the following: 
 
 \begin{proposition}
 Let $R$ be a ring. Then every cotorsion module is $\mathscr{T}$-injective if and only if all pure projective tilting modules are projective, and all $R$-modules are $\mathscr{T}$-injective.
\end{proposition}

\begin{proof}
 First we assume that every cotorsion module is $\mathscr{T}$-injective. Let $M$ be an arbitrary $R$-module. Then we consider the exact sequence $0 \rightarrow M \rightarrow C(M) \stackrel{g}{\rightarrow} C(M)/M \rightarrow 0,$ where $C(M)$ is a cotorsion envelope of $M$ and $C(M)/M$ is a flat $R$-module. Then the sequence is pure exact. For any pure projective tilting $R$-module $T$ and applying the functor $\Hom_R(T, -),$ we get the exact sequence \[\Hom_R(T, C(M)) \stackrel{g^\star}{\rightarrow} \Hom_R(T, C(M)/M) \rightarrow \Ext_R^1(T, M) \rightarrow \Ext_R^1(T, C(M)).\] By hypothesis,  $\Ext_R(T, C(M)) = 0.$ It follows that $\Ext_R^1(T, M) = 0$ since $g^\star$ is surjective. This implies that $T$ is projective and all $R$-modules are $\mathscr{T}$-injective. Converse part is clear.
 \end{proof}
 
 \begin{proposition}
 Let $R$ be a Noetherian ring. Then the  following conditions are equivalent:
 \begin{enumerate}
     \item [(1)] $R$ is a  pure-tilting hereditary ring and every cotorsion module is $\mathscr{T}$-injective;
     \item [(2)] Every pure projective tilting module is injective.
 \end{enumerate}
 \end{proposition}
 
 \begin{proof}
  $(1) \Rightarrow (2)$ is clear.\\
  
  $(2) \Rightarrow (1).$ Let $I$ be a finitely generated ideal of $R.$ Hence the exact sequence $0 \rightarrow I \rightarrow R \rightarrow R/I \rightarrow 0$ splits because $I$ is injective. It follows that every finitely generated ideal is a direct summand of $R.$ Then $R$ is a von Neumann regular ring.
 \end{proof}
 

\begin{proposition} \label{PROP:4.4}
Let $R$ be a  pure-tilting hereditary ring. Then every $\mathscr{T}$-injective $R$-module has injective dimension at most $1.$
\end{proposition}
\begin{proof}
 Suppose $A$ is $\mathscr{T}$-injective, that is, $\Ext_R^i(G, A) = 0$ for all $G \in \mathscr{T}$ and $i \geq 1.$ For any $I \subseteq R,$ consider an exact sequence $0 \rightarrow I \rightarrow R \rightarrow R/I \rightarrow 0.$ This implies that the exact sequence $0 \rightarrow \Ext_R^1(I, A) \rightarrow \Ext_R^2(R/I, A) \rightarrow 0.$ Therefore $\Ext_R^1(I, A) = 0$ since $R$ is pure hereditary and also all pure projective $n$-tilting modules. Hence $\id(A) \leq 1.$ 
\end{proof}

\begin{proposition}\label{6.3.1}
The class $\mathscr{T}^{\bot}$ of all $\mathscr{T}$-injective modules is closed under pure submodules.
\end{proposition}

\begin{proof}
Let $A$ be a pure submodule of a $\mathscr{T}$-injective module $G$. Then there is a pure exact sequence $0 \rightarrow A \rightarrow G \rightarrow M/A \rightarrow 0$ and a functor $\Hom_R(T, -)$ preserves this sequence is exact whenever $G \in \mathscr{T}.$ This implies that the sequence $0 \rightarrow \Hom_R(T, A) \rightarrow \Hom_R(T, G) \rightarrow \Hom_R(T, G/A) \rightarrow \Ext_R^1(T, A) \rightarrow 0$ is also exact for all $T \in \mathscr{T}$. It follows that $\Ext_R^1(T, A) = 0$ for all $T \in \mathscr{T}$. 
\end{proof}

Clearly, the class of all $\mathcal{T}$-injective modules is closed under direct products. By the above Proposition \ref{6.3.1}, the class of all $\mathscr{T}$-injective modules is closed under pure submodules. By \cite[Corollary ~6.5]{PP:2005}, pure projective tilting module is a direct sum of finitely presented modules over a hereditary noetherian ring. By \cite[Lemma ~3.1.6, Lemma ~3.1.10]{tri1}, we immediately get the following remark.

\begin{remark}
Let $R$ be a  hereditary noetherian ring and let $\mathcal{T}$ be the class of all classical pure projective tilting modules. Then the class of all $\mathscr{T}$-injective modules is definable. Consequently, a module is $\mathscr{T}$-injective if and only if its pure injective envelope is $\mathscr{T}$-injective.
\end{remark}

\begin{proposition}\label{4.2.0}
Let $R$ be a  pure-tilting hereditary ring. Then the class $\mathscr{T}^{\bot}$ of all $\mathscr{T}$-injective modules is coresolving.
\end{proposition}

\begin{proof}
Let $0 \rightarrow M_{1} \stackrel{\phi}{\rightarrow} M_{2} \stackrel{\psi}{\rightarrow} M_{3} \rightarrow 0$ be an exact sequence of left $R$-modules with $M_{1}, M_{2} \in  \mathscr{T}^{\bot}$. Let $G \in \mathscr{T}.$ Then the associated the long exact sequence of $\Ext$ groups yields, for each $i\geq 1,$ $\Ext_R^i(G, M_2) \rightarrow \Ext_R^i(G, M_3) \rightarrow \Ext_R^{i+1}(G, M_1)$ is exact. By Proposition \ref{PROP:4.4} and the definition of $\mathscr{T}$-injective modules, we have, $\Ext_R^i(G, M_1) = \Ext_R^i(G, M_2) = 0$ for all $i \geq 1$ and all $G \in \mathscr{T}.$ It follows immediately from the above exact sequence that $\Ext_R^i(G, M_3) = 0$ for all $i \geq 1$ and all $G \in \mathscr{T}$. Hence $M_{3}$ is $\mathscr{T}$-injective. This completes the proof.
\end{proof}

\begin{proposition}\label{4.2-1}
Let $R$ be a  pure-tilting hereditary ring. Then every factor module of an $\mathscr{T}$-injective $R$-module by a pure submodule is $\mathscr{T}$-injective.
\end{proposition}

\begin{proof}
 Let $A$ be a pure submodule of $\mathscr{T}$-injective $R$-module $B.$ By Proposition \ref{6.3.1} and Proposition \ref{4.2.0}, $B/A$ is $\mathscr{T}$-injective.
\end{proof}

The study of preenvelopes and precovers originated in the work of E. Enochs, who introduced these concepts in connection with approximation theory and relative homological algebra. The existence of preenvelopes plays an important role in relative homological algebra and approximation theory. The following theorem establishes the existence of $\mathcal{T}$-injective preenvelopes in our present setting.

\begin{theorem}\label{6.3.2}
Let $R$ be a  pure-tilting hereditary ring. Then every $R$-module has an $\mathscr{T}$-injective preenvelope. \end{theorem}

\begin{proof}
Let $\mathcal{C}=\mathcal{T}^{\perp}$ be the class of all $\mathcal{T}$-injective $R$-modules. By \cite[Lemma ~5.3.12]{enochs}, there exists a cardinal $\aleph_{\alpha}$ such that for every homomorphism $\phi : M \to G,$ where $G\in\mathcal{C}$, there exists a pure submodule $A\subseteq G$ satisfying $\phi(M)\subseteq A \quad\text{and}\quad |A|\leq\aleph_{\alpha}.$

Since $\mathcal{C}$ is closed under arbitrary direct products and, by Proposition 4.5, pure submodules of modules in $\mathcal{C}$ again belong to $\mathcal{C}$, it follows that $A\in\mathcal{C}$. Hence every morphism $\phi:M\to G$ factors through some module $A\in\mathcal{C}$ with cardinality at most $\aleph_{\alpha}$.

Therefore, the class $\mathcal{C}$ satisfies the assumptions of [9, Proposition 6.2.1]. Consequently, $\mathcal{C}$ is preenveloping. Equivalently, every $R$-module admits a $\mathcal{T}$-injective preenvelope.
\end{proof}

\begin{remark} \label{REMK:5.1}
Let $R$ be a  pure-tilting hereditary ring. Then, by Theorem \ref{6.3.2}, every module has a special  $\mathscr{T}^\bot$-preenvelope. By \cite[Salce Lemma ~2.2.6]{tri1}, every module has a $^\bot(\mathscr{T}^\bot)$-precover.
\end{remark}
\begin{remark}\label{REMK:5.2}
Let $R$ be a von Neumann regular ring. Then every $R$-module is $\mathscr{T}$-injective. Further, every pure-tilting hereditary ring is hereditary.
\end{remark}

\begin{proof}
 By Theorem \ref{6.3.2}, we consider a short exact sequence $0 \rightarrow M \rightarrow T(E) \rightarrow T(E)/M \rightarrow,$ where $T(E)$ is an $\mathscr{T}$-injective preenvelope of $M.$ This sequence is a pure exact sequence since $R$ is von Neumann regular. It follows that $\Hom_R(T^\prime, T(E) \rightarrow \Hom_R(T^\prime, T(E)/M)) \rightarrow 0$ is surjective for all pure projective tilting $R$-modules $T^\prime.$ Consequently, $\Ext_R^1(T^\prime, M) = 0.$ Thus, Every $R$-module is $\mathscr{T}$-injective and hence $\mathscr{T}^{\bot} = R \mbox{-} \Mod.$ It follows that every pure projective tilting module is projective. Therefore, every pure-tilting hereditary ring is hereditary. 
\end{proof}

The converse statement in Remark \ref{REMK:5.2} is valid provided that the class of pure classical tilting modules coincides with the class of finitely generated modules.

\begin{proposition}
 Let $R$ be a  pure-tilting hereditary ring. If every $R$-module has a special $\mathscr{T}$-injective cover, then $R$ is $\mathscr{T}$-injective.
\end{proposition} 

\begin{proof}
  
 Assume that every module has a special $\mathscr{T}$-injective cover. We consider the following exact sequence $\epsilon \colon 0 \rightarrow K \rightarrow G \rightarrow S \rightarrow 0$, where $G$ is an $\mathscr{T}$-injective module and $K = \ker (K \rightarrow G).$ Clearly, $\epsilon$ is a pure exact sequence. By Proposition \ref{6.3.1}, $K$ is $\mathscr{T}$-injective, and therefore $S$ is $\mathscr{T}$-injective by Proposition \ref{4.2.0}.  
\end{proof}
 
\begin{theorem}
Let $R$ be a  pure-tilting hereditary ring. Then the following conditions are hold:
\begin{enumerate}
    \item Every factor module of an $\mathscr{T}$-injective module by an absolutely pure submodule is $\mathscr{T}$-injective.
    \item Every absolutely pure module is $\mathscr{T}$-injective.
    \item Every absolutely pure module has injective dimension at most $1.$
    \end{enumerate}
\end{theorem}

\begin{proof}
 $(1).$ Let $A$ be an absolutely pure module. By Theorem \ref{6.3.2}, there is an exact sequence $\eta: 0 \rightarrow A \rightarrow G \rightarrow G/A \rightarrow 0$ of $R$-modules with $A$ is absolutely pure and $G$ is $\mathscr{T}$-injective. The above sequence $\eta$ is pure exact since $A$ is absolutely pure. Hence $G/A$ is $\mathscr{T}$-injective by Proposition \ref{4.2-1}.
 
    $(2).$ By Proposition \ref{6.3.1}, $A$ is $\mathscr{T}$-injective from the sequence $\eta.$
 
    $(3).$ Consider an exact sequence $0 \rightarrow I \rightarrow R \rightarrow R/I \rightarrow 0.$ For any absolutely module $A,$ the sequence $0 \rightarrow \Ext_R^1(I, A) \rightarrow \Ext_R^2(R/I, A) \rightarrow 0$ is exact. From $(2), A$ is $\mathscr{T}$-injective and hence $\Ext_R^1(I, A) = 0$ since $R$ is pure-tilting hereditary . It follows that $\Ext_R^2(R/I, A) = 0$ since $R$ is pure-tilting hereditary. Hence, $\id_R(A) \leq 1.$
\end{proof}

\section{Gorenstein pure projective titling modules}

In this section, we introduce the Gorenstein pure projective tilting modules and discuss the main result of the section, the connection between Gorenstein pure projective tilting modules and strict stationary modules.

An $R$-module $M$ is called \textit{Gorenstein pure projective tilting} provided that $M$ occurs as a syzygy in an exact complex of projective modules 
\begin{equation}
   \mathrm{P}_\bullet \colon \cdots \rightarrow P_1 \rightarrow P_0 \rightarrow P^0 \rightarrow P^1 \cdots 
\end{equation}
which remains exact after applying the functor $\Hom_R(-, T)$ for all $T \in \mathscr{T}.$ 

 An immediate consequence is that, by symmetry of the complex $\mathrm{P}_\bullet$ arising from a complete projective resolution, every syzygy is a Gorenstein pure projective tilting module. It follows that all kernels and cokernels appearing in the complex are also Gorenstein pure projective tilting. Moreover, every projective module is, in particular, Gorenstein pure projective tilting. Furthermore, every Gorenstein projective module is Gorenstein pure projective $0$-tilting, though not necessarily Gorenstein pure projective tilting, since pure projective tilting modules are not, in general, projective.
We now ask when the converse statement holds. Suppose that all modules are $\mathscr{T}$-injective. Then $\mathscr{T}^\perp = R\text{-}\mathrm{Mod}$, and consequently every pure projective tilting module is projective. By Lemma~\ref{REMK:5.2}, over a von Neumann regular ring, every Gorenstein projective module is Gorenstein pure projective projective tilting. Hence, a natural question is under what conditions the converse implication is valid.
Suppose further that every projective module is a progenerator (and hence tilting). Under this assumption, the converse holds; that is, the class of pure projective tilting modules coincides with the class of projective modules. In particular, it is automatically satisfied over a simple Artinian ring, since every projective module is a generator in this setting. Consequently, we obtain the following result.

 \begin{proposition}
Let $R$ be a simple Artinian ring. Then an $R$-module $M$ is Goresntein pure projective tilting if and only if it is Goresntein projective. 
 \end{proposition}

By definition, we immediately obtain the following proposition, which will be used in the proof of the main theorem of this section.

\begin{proposition} \label{PRO:6.2}
    An $R$-module $M$ is Gorenstein pure projective tilting if and only if $M$ belongs to the left orthogonal class $^\perp\mathscr{T}$ and $M$ admits a co-proper right projective resolution. 
\end{proposition}

\begin{example}
     Let $R = \mathbb{Z}$ and let $T = \mathbb{Z}.$ Since  $\mathbb{Z}$ is projective and hence a pure projective $1$-tilting $\mathbb{Z}$-module, $\mathbb{Z}\in \mathscr{T}$. Now we consider the exact complex $\mathbf{P}_\bullet \colon \cdots \rightarrow 0 \rightarrow \mathbb{Z} \xrightarrow{1} \mathbb{Z} \rightarrow 0 \rightarrow \cdots$ which is exact and consists of projective modules. Applying the functor $\Hom_R(-, \mathbb{Z})$ preserves exactness because $\mathbb{Z}$ is projective. Hence $\mathbb{Z}$ is a trivial example for Gorenstein pure projective tilting module. 
 \end{example}

 \begin{example}
     Let $R=k[x]/(x^2)$, where $k$ is a field. We consider a simple module $S = R/(x).$ Since $R$ is a local non-semisimple Artin algebra, $S$ is not projective. Nevertheless, $S$ is Goresntein projective. Clearly, $S$ admits the complete projective resolution $\mathbf{P}_\bullet \colon \cdots \xrightarrow{x} R \xrightarrow{x} R \xrightarrow{x} R \xrightarrow{x} \cdots .$ Because $x^2 = 0,$ we have $\ker (x) = xR \cong R/(x) = S$ and $\im (x) = xR$ and hence the complex $\mathbf{P}_\bullet$ is exact. Now, let $T$ be any pure injective tilting module. Applying the functor $\Hom_R(-, T)$ to $\mathbf{P}_\bullet$ yield the complex $\cdots \xleftarrow{x} T \xleftarrow{x} T \xleftarrow{x} T \xleftarrow{x} T \leftarrow \cdots.$ Since $R$ is a self-injective Artin algebra, every totally acyclic projective complexes remains exact after applying $\Hom_R(-, M)$ for every module $M$ of finite projective dimension. In particular, this holds for tilting modules. Hence $\Hom_R(\mathbf{P}_\bullet, T)$ is exact. Consequently, $S = R/(x)$ is Gorenstein pure projectie tilting module, but it is not projective. 
 \end{example}

\begin{example}
    Consider the cyclic quiver $Q \colon 1 \xrightarrow{\alpha} 2 \xrightarrow{\beta} 1$ with relations $\alpha \beta = 0$ and $\beta \alpha = 0.$ Then $R = kQ/(\alpha \beta, \beta \alpha)$ is a self injective Nakayama Algebra. Denote by $S_1$ and $S_2$ the simple $R$-modules corresponding to the vertices $1$ and $2,$ respectively. Neither $S_1$ nor $S_2$ is projective.  Each simple module admits a periodic projective resolution. We consider a minimal projective resolution $\cdots \rightarrow P_1 \rightarrow P_2\rightarrow P_1 \rightarrow P_2 \rightarrow S_1 \rightarrow0$ and similarly for $S_2$. Since the resolutions are periodic, they extend naturally to bi-infinite exact complexes of projective modules $\mathbf{P}_\bullet \colon \cdots  P_1 \rightarrow P_2\rightarrow P_1 \rightarrow P_2 \rightarrow \cdots$ whose cycles module is precisely the simple modules $S_1$ and $S_2$. Hence $\mathbf{P}_\bullet$ is a totally acyclic complex of projective modules. Since $R$ is self-injective, $\mathbf{P}_\bullet$ remains exact under $\Hom_R(-, T)$ for every pure projective tilting module $T.$ Hence each simple module $S_i (i=1,2)$ is Gorenstein pure projective tilting.
\end{example}

We prove the following main theorem in the section

\begin{theorem}\label{MAIN:GOREN}
Let $\mathscr{T}$ be the class of all finitely presented tilting modules. Then every Gorenstein pure projective tilting module is Gorenstein flat if and only if every Gorenstein pure projective tilting module is a strict $\mathscr{T}$-stationary module. 
\end{theorem}

\begin{proof}
Let $M$ be a Gorenstein pure projective tilting module and let $T\in \mathcal T$. Since $M$ is Gorenstein pure projective, there exists a complete projective resolution $\mathbf P=\cdots \longrightarrow P_1 \longrightarrow P_0
\longrightarrow P_{-1}\longrightarrow \cdots$ such that $M \cong \Omega^1(\mathbf P) = \ker(P_0\to P_{-1}).$ Put $N=\ker(P_{-1}\to P_{-2})$. Then $M\cong \Omega^1(N)$ and $N$ is Gorenstein projective. By assumption, every Gorenstein pure projective tilting module is Gorenstein flat; hence $N$ is Gorenstein flat. Since $T$ is a tilting module, there exists an integer $n\geq 0$ such that $\pd_R(T)\leq n$. Therefore the character module $T^+=\Hom_{\mathbb Z}(T,\mathbb Q/\mathbb Z)$ satisfies $\id_R(T^+) \leq n.$
As $N$ is Gorenstein flat, we have $\Tor_i^R(T^+,N)=0 \qquad \text{for all } i\geq 1.$
Hence, by Proposition~2.4, the syzygy module $M = \Omega^1(N)$ is strict $\mathcal T$-stationary. Therefore, every Gorenstein pure projective tilting module is strict $\mathcal T$-stationary. Conversely, suppose that every Gorenstein pure projective tilting module is strict $\mathcal T$-stationary. Let $\mathbf{P}_\bullet \colon \cdots  P_1 \rightarrow P_2\rightarrow P_{-1} \rightarrow P_{-2} \rightarrow \cdots$ be a complete projective resolution. The $i^{th}$-syzygy of the complex $\Omega^i(\mathbf{P}_\bullet)$ is $\ker(P_i \rightarrow P_{i-1}).$ Assume that every syzygy $\Omega^i(\mathbf{P}_\bullet)$ is Gorenstein pure projective tilting. By hypothesis, each $\Omega^i(\mathbf{P}_\bullet)$ is therefore a strict $\mathscr{T}$-stationary module. We now show that each $i,$ $\Omega^i(\mathbf{P}_\bullet)$ is Gorenstein flat. Equivalently, it suffices to prove that the complex $E \otimes \mathbf{P}_\bullet$ is exact for every injective $R$-module $E.$ By standard homological criteria, this is equivalent to showing that $\Tor^R_1(E, \Omega^i(M)) = 0$ for every $i \geq 1.$ Since all syzygies occurring in a complete projective resolution satisfy the same homological properties, it is enough to verify the assertion for a single syzygy, say $\Omega^i(\mathbf{P}_\bullet).$ Consequently proving $\Tor^R_1(E, \Omega^i(\mathbf{P}_\bullet)) = 0$ implies the corresponding vanishing for every $\Omega^i(\mathbf{P}_\bullet),$ and hence establishes that all syzygies are Gorenstein flat. Fix $i \in \mathbb{Z}.$ By \cite[Corollary~8.5]{LH:2008}, $\Omega^i(\mathbf P)$ is strict $\Add(T)$-stationary. Hence, by Proposition~\ref{P1}, for every $T'\in \Add(T)$, the canonical map
\[
\Phi_{\Omega^i(\mathbf{P}_\bullet)}^{(i)}:
\Tor_i^R(\Hom_R(T',\mathbb Q/\mathbb Z),\Omega^i(\mathbf{P}_\bullet)))
\longrightarrow
\Hom_{\mathbb Z}(\Ext_R^i(\Omega^i(\mathbf{P}_\bullet)),T'),\mathbb Q/\mathbb Z)
\]
is injective. Moreover, every module in $\Add (T)$ is a pure projective tilting. Hence, $\Ext_R^m(\Omega^i(\mathbf{P}_\bullet)), T^\prime) = 0$ for all $m \geq 1$ by Proposition \ref{PRO:6.2}. Hence, $\Ext_R^m(\Omega^i(\mathbf{P}_\bullet)), T^\prime) = 0$ for all $m \geq 1$ and all $T^\prime \in \Add (T).$ On the other hand, since $T$ is tilting, there are $r \geq 0$ and a long exact sequence \[0 \longrightarrow R \longrightarrow T_0 \longrightarrow T_1 \longrightarrow \cdots \longrightarrow T_r \longrightarrow 0,\] where $T_i \in \Add (T)$ for all $i \leq r.$ Applying the character module functor $(-)^+$ yields the exact sequence 
\begin{equation} \label{eq:dual-sequence}
    0 \longrightarrow T_r^+ \longrightarrow T_{r-1}^+ \longrightarrow \cdots \longrightarrow T_0^+ \longrightarrow R^+ \longrightarrow 0,
\end{equation}
where $T_m^+ \in \Prod (T^+)$ for each $0 \leq m \leq r.$ Furthermore, for any cardinal $\kappa,$ we have \[\Tor^R_i((T^+)^\kappa, \Omega^i(\mathbf{P}_\bullet)) \cong \Tor^R_i((T^{(\kappa)})^+, \Omega^i(\mathbf{P}_\bullet)) = 0.\] Therefore, since $T_m^+ \in \Prod (T^+)$ for every $0 \leq m \leq r,$ it follows that $\Tor^R_1(T_m^+, \Omega^i(\mathbf{P}_\bullet)) = 0.$ Taking $\kappa$-fold products in \eqref{eq:dual-sequence}, we obtain the exact sequence \[0 \longrightarrow (T_r^+)^\kappa \longrightarrow (T_{r-1}^+)^\kappa \longrightarrow \cdots \longrightarrow (T_0^+)^\kappa \longrightarrow (R^+)^\kappa \longrightarrow 0,\] Since $\Omega^i(\mathbf{P}_\bullet))$ is a strict $\mathscr{T}$-stationary module, it follows from Lemma \ref{P3} that $\Omega^i(\mathbf{P}_\bullet))$ is a $\mathcal{C}$-Mittag-Leffler module, where $\mathcal{C}$ is the class of all cotilting modules (that is, character modules of pure projective tilting modules). It is clear that $\Omega^i(\mathbf P)$ is also a $\Prod (T^+)$-Mittag-Leffler module. This shows that $\Omega^i(\mathbf P)$ is $T_m$-Mittag-Leffler module for any $0 \leq m \leq r.$ By Lemma \ref{P2}, we have \[\eta \colon \Tor^R_{i+1}((T_m^+)^\kappa, \Omega^i(\mathbf P)) \cong (\Tor^R_{i+1}(T_m^+, \Omega^i(\mathbf P))^{(\kappa)} = 0.\] is injective. Hence, $\Tor^R_1((T_m^+)^\kappa = 0$ for any $0 \leq m \leq r.$ By the definition of Gorenstein pure projective tilting modules, it is easy to check that \[\Tor^R_1((R^+)^\kappa, \Omega^i(\mathbf P)) \cong \Tor^R_1((T_r)^\kappa, \Omega^i(\mathbf P)) = 0.\] Since any injective $R$-module is a direct summand of a direct product of copies of $R^+,$ we have that $\Tor^R_1(E, \Omega^i(\mathbf P)) = 0$ for any injective $R$-module $E.$ This proves that every Gorenstein pure projective tilting module is Gorenstein flat. 
\end{proof}

\section*{acknowledgment}
    The author gratefully acknowledges the financial support provided by the NBHM grant (No. 02011/14/2025 NBHM (R. P)/R\&D II/2172) and partial support from the Prof. T. R. Rajagopalan Start-up Grant (SASTRA-TRR-SASHE-2-25012025). The author also thanks Institut des Hautes Études Scientifiques (IHES), Université Paris-Saclay, for providing a stimulating research environment during the preparation of this work, and the European Mathematical Society for supporting the visit to IHES. The author is sincerely grateful to Prof. Keller for his valuable suggestions and insightful comments, which significantly improved the quality of this article.


\begin{thebibliography}{99}

\bibitem{and}
F. ~W. Anderson \and K. ~R. Fuller, {\it Rings and Categories of Modules}, Springer-Verlag, New York, 2nd ed. (1992).

  

\bibitem{LA}
L. ~Angeleri H\"ugel, {\it On some precovers and preenvelopes}, 2000.

\bibitem{LH:2008}
L. ~Angeleri H\"ugel \and D. ~Herbera, Mittag-Leffler conditions on modules, {\it Indiana Univ. Math. J.}, 57 (2008) 2459--2517.

\bibitem{ASS:2010}
I. ~Assem,  A. ~Skowronski \and D. ~Simson, Elements of the Representation Theory of Associative Algebras, {\it London Math. Soc.}, 65 (2010).

\bibitem{BH} 
 S. Bazzoni, I. Herzog, P. Příhoda, J. Šaroch, \and J. Trlifaj, Pure projective tilting modules, {\it  Documenta Math.},  25 (2020) 401--424.

\bibitem{Ding} 
N. ~Q. Ding, On envelopes with the unique mapping property, {\it Comm. Algebra}, 24(4) (1996) 1459--1470.

\bibitem{ET:2011}
 I. ~Emmanouil \and O. ~Talelli, On the flat length of injective modules,{\it J. Lond. Math. Soc.}, 84 (2011) 408--432.

\bibitem{Eno0}
E. ~E. ~Enochs, A note on absolutely pure modules, {\it Canad. Math. Bull.}, 19 (3), (1976) 361--362.

\bibitem{Eno1} 
E. ~E. ~Enochs, Injective and flat covers, envelopes and resolvents, {\it  Isrel J. Math.}, 39 (1981) 189--209.

\bibitem{enochs} 
E. ~E. ~Enochs \and O. ~M. ~G. ~Jenda, {\it Relative Homological Algebra}, de Gruyter Exp. Math. 30, Walter de Gruyter: Berlin-New York, (2011).

\bibitem{FS:1974}
J. ~W. Fisher \and R. ~L. Snider, On the Von Neumann Regularity of rings with regular prime fator rings, {\it Pacifi J. Math.}, 54(1) (1974) 135--144.

\bibitem{GJ:1975} 
 S. ~C. ~Goel, S. ~K. ~Jain \and S ~Singh, Rings whose cyclic modules are injective or projective, {\it Proc. Amer. Math. Soc.}, 53 (1) (1975) 16--18.

\bibitem{RGT:2006} 
R. ~G\"obel \and J. ~Trlifaj, {\it Approximations and Endomorphisms Algebras of Modules}, de Gruyter, Berlin, (2000).

\bibitem{HM:2004} 
H. ~Holm, Gorenstein homological dimensions, {\it J. Pure Appl. Algebra}, 189 (2004), 167--193.

\bibitem{KAP:1970}
I. ~Kaplansky, {\it Algebraic and Analytic Aspects of Operator Algebras}, Amer. Math. Soc, Providance, R. I., (1970).

\bibitem{LM:1999} 
T. ~Y. ~Lam, {\it Lectures on Modules and Rings}, Graduate Texts in Mathematics, 189, Springer-Verlag, New York, 1999.

\bibitem{mao1} 
L. ~Mao \and N. ~Q. ~Ding, $\mathcal{L}$-injective hulls of modules, {\it Bull. Aus. Math. Soc.}, 74 (2006) 37--44.

\bibitem{CM}
C. ~Megibben, Absolutely pure modules, {\it Proc. Amer. Math. Soc.}, 26, (1970) 561-566.

\bibitem{MD:2017}
A. ~Moradzadeh-Dehkordi \and S. ~H. ~Shojee,
Rings in which every ideal is pure-projective or FP-projective, {\it J. Algebra}, 478 (2017) 419--436.

\bibitem{MD}
A. ~Moradzadeh-Dehkordi,
On the structure of pure-projective modules and some applications, {\it J. Pure Appl. Algebra}, 221 (2017) 935--947.

\bibitem{OS:1971}
U. ~Oberst \and H. ~J. ~Schneider, Die Struktur von projektiven Moduln, {\it Invent. Math.}, 13 (1971) 295-304.

\bibitem{pru} 
H. ~Pr\"ufer, Studies on the dismantling of countable primary abelian groups, {\it Math. Zeit.}, 17(1) (1923) 35--61.

\bibitem{PP:2005} 
G. ~Puninski \and P. ~Rothmaler, pure projective modules, {\it J. London Math. Soc.}, 71(2) (2005) 304--320.

\bibitem{rot} 
J. ~J. ~Rotman, {\it An Introduction to Homological Algebra}, Academic Press, New York, (1979).

\bibitem{uma} 
A. ~Umamaheswaran, R. ~Udhayakumar, C. ~Selvaraj, K. ~Tamilvanan, \and M. ~J. ~Kabeto,  Existence of covers and envelopes of a left orthogonal class and its right orthogonal class of modules, {\it Journal of Mathematics}, 2022 (2023) 12 pages.

\bibitem{sim} 
D. ~Simson, Pure-periodic modules and a structure of pure-projective resolutions, {\it Pacific journal of Mathematics}, 207 (1) (2002) 235--256. 

\bibitem{stns} 
 B. ~Stenstr\"om, Coherent Rings and $\fp$-injective modules, {\it J. London Math. Soc.}, 2 (1970), 323--329.

\bibitem{BOS:1975}
B. ~Stenstr\"om, {\it Grothendieck Categories. In: Rings of Quotients. Die Grundlehren der mathematischen Wissenschaften}, Springer, Berlin, Heidelberg, 217 (1975).

\bibitem{tri1} 
J. Trlifaj, {\it Infinite dimensional tilting modules and cotorsion pairs}, Hand book of tilting theory, Lect. Notes, 322, London Math. Soc., Cambridge, (2007) 279--321.

\bibitem{WAR:1969}
  R.~B.~Warfield, Purity and algebraic compactness for modules, {\it Pacific. J. Math.}, 28 (1969) 699-719.

\bibitem{xu} 
 J. Xu, {\it Flat covers of Modules}, Lecture Notes in Mathematics. Vol 1634 Springer-Verlag, Germany, (1996).

\bibitem{lam}
T. ~Y. ~Lam, {\it Lectures on Modules and Rings}, Graduate Texts in
Math., 189 (Springer-Verlag, Berlin, New York, Heidelberg), (1999).
\bibitem{vam}
P. Vamos, Rings with duality, {\it Proc. London Math. Soc.}, {\bf 35}
(1977) 275--289.

\end{thebibliography}
\end{document}